\documentclass[a4paper,twoside,DIV=16,headings=small,headsepline=true]{scrartcl}
\pdfoutput=1
\usepackage{amsmath}
\usepackage{amsfonts}
\usepackage{amssymb}
\usepackage{amsthm}
\usepackage[T1,T2A]{fontenc}
\usepackage[utf8]{inputenc}
\usepackage{fourier}
\RequirePackage[scaled=0.9]{PTSansCaption}
\RequirePackage[scaled=0.9]{PTSerifCaption}
\RequirePackage[scaled=0.9]{PTSans}
\RequirePackage[scaled=0.9]{PTSerif}
\usepackage[selected=true,shrink=5,stretch=40,letterspace=30]{microtype}
\usepackage[inner=1.6cm,outer=2.4cm,
top=1.6cm,bottom=1.6cm,footskip=.5cm,headsep=.3cm]{geometry}

\usepackage{extarrows}
\usepackage{colonequals}
\usepackage{mathtools}
\usepackage[british,russian]{babel}

\usepackage{paralist}
\usepackage{setspace}
\usepackage{xcolor}

\newcommand{\sign}{\operatorname{sign}}
\newcommand{\lc}{\operatorname{\mathbf{lc}}}
\newcommand{\zr}{\operatorname{\mathbf{zr}}}
\renewcommand{\ge}{\geqslant}
\renewcommand{\le}{\leqslant}

\makeatletter
\def\namedlabel#1#2{\begingroup
   \def\@currentlabel{{\sffamily #2}}%
   \label{#1}\endgroup
}
\makeatother

\usepackage[
        colorlinks=false, pdfborder={0 0 .5},
        citecolor=DarkGreen,linktoc=all,unicode=true,
	pdftitle={On conjectures by Csordas, Charalambides and Waleffe},
	pdfauthor={Alexander Dyachenko, Galina van Bevern},
	pdfsubject={MSC(2010) 33C45, 26C10, 30C15},
	pdfkeywords={Jacobi polynomials, Interlacing zeros,
            Tau methods, Hurwitz stability, Hermite-Biehler theorem},
	pdfpagemode=UseOutlines,pdflang=en-GB,
        pdfa=true]{hyperref}

\newcounter{dummy}

\title{On conjectures by Csordas, Charalambides and Waleffe%
    \thanks{This work was financially supported by the European Research Council under the
        European Union's Seventh Framework Programme (FP7/2007--2013)/ERC grant agreement no.
        259173.}}
\author{\fontfamily{PTSansCaption-TLF}\selectfont\large Alexander Dyachenko \and \fontfamily{PTSansCaption-TLF}\selectfont\large Galina van Bevern}
\date{\fontfamily{PTSansCaption-TLF}\selectfont\large \vskip .3em 12th June 2015\vskip -2em}

\addtokomafont{title}{\fontfamily{PTSansCaption-TLF}\selectfont}
\addtokomafont{pagehead}{\fontfamily{PTSans-TLF}\selectfont}
\addtokomafont{section}{\fontfamily{PTSansCaption-TLF}\selectfont}
\addtokomafont{subsection}{\fontsize{1em}{1.1em}\fontfamily{PTSansCaption-TLF}\selectfont}

\usepackage{scrpage2}
\ohead{\pagemark}
\ihead{\headmark}
\cehead{On conjectures by Csordas, Charalambides and Waleffe}
\cohead{Alexander Dyachenko and Galina van Bevern}
\ifoot{}\ofoot{}
\setheadsepline{.4pt}

\begin{document}
\selectlanguage{british}
\maketitle
\begin{abstract}
    In the present note we obtain new results on two conjectures by Csordas et al. regarding the
    interlacing property of zeros of special polynomials. These polynomials came from the Jacobi
    tau methods for the Sturm-Liouville eigenvalue problem. Their coefficients are the
    successive even derivatives of the Jacobi polynomials~$P_n^{(\alpha,\beta)}$ evaluated at
    the point one. The first conjecture states that the polynomials constructed from
    $P_n^{(\alpha,\beta)}$ and $P_{n-1}^{(\alpha,\beta)}$ are interlacing when $-1<\alpha<1$ and
    $-1<\beta$. We prove it in a range of parameters wider than that given earlier by
    Charalambides and Waleffe. We also show that within narrower bounds another conjecture
    holds. It asserts that the polynomials constructed from $P_n^{(\alpha,\beta)}$ and
    $P_{n-2}^{(\alpha,\beta)}$ are also interlacing.
\end{abstract}

{\raggedright \fontsize{9pt}{10pt} \fontfamily{PTSansCaption-TLF}\selectfont
    }

\newtheoremstyle{mytheoremstyle}{4pt}{4pt}{\itshape}{}{\bfseries\sffamily}{.}{.5em}{}
\newtheoremstyle{mydefinitionstyle}{4pt}{4pt}{}{}{\bfseries\sffamily}{.}{.5em}{}
\newtheoremstyle{citing}{4pt}{4pt}{\itshape}{}{\bfseries\sffamily}{.}{.5em}{\thmnote{#3}}

\theoremstyle{citing}
\newtheorem*{theorem*}{}
\theoremstyle{mytheoremstyle}
\newtheorem{theorem}{Theorem}
\newtheorem{lemma}[theorem]{Lemma}
\newtheorem{corollary}[theorem]{Corollary}

\newtheorem{con}{Conjecture}

\theoremstyle{mydefinitionstyle}
\newtheorem*{rem}{Remark}
\newtheorem{ob}{Observation}
\newtheorem{ex}{Example}
\newtheorem*{definition}{Definition}

\renewcommand{\Re}{\operatorname{Re}}
\renewcommand{\Im}{\operatorname{Im}}

\section{Introduction}
This study is devoted to properties of zeros of polynomials which originate from mathematical
physics. The orthogonal polynomials proved to be a very helpful tool for the discretization of
linear differential operators. The main feature of the tau methods is the adoption of a
polynomial basis which does not automatically satisfy the boundary conditions. This induces a
problem at the boundary (i.e. the points~$\pm1$ in the Jacobi case). The family of polynomials
studied here is connected this way to the eigenproblem $u''(x)=\lambda u(x)$ on the
interval~$x\in(-1,1)$ with various homogeneous boundary conditions (for the details
see~\cite{Cs,CW}). We place our main emphasis on the analytic properties of the considered
family itself, leaving aside the corresponding properties of the original differential
operators. More information on the tau methods can be found in e.g.~\cite[\S10.4.2]{Ca}.

The \emph{Jacobi polynomials} (see their definition and basic properties in
e.g.~\cite[Ch.~IV]{Sz})
\[P_{n}^{(\alpha,\beta)}(x)
=\binom {n+\alpha}{ n}\; {}_{2\!}F_{1}
\left [ \begin{matrix} -n,&n+\alpha+\beta+1\\ \alpha+1 \end{matrix}; \frac{1-x}{2} \right],
\quad
n=1,2,\dots
\]
appear regularly in applications as classical orthogonal polynomials. They are more general than
those of Chebyshev, Legendre and Gegenbauer. The Jacobi polynomials are orthogonal with respect
to the measure~$w_{\alpha,\beta}(x)=(1-x)^\alpha (1+x)^\beta$ on the interval~$(-1,1)$ whenever
both the parameters~$\alpha$ and $\beta$ are greater than~$-1$:
\[
\int_{-1}^1 P^{(\alpha,\beta)}_n(x)P^{(\alpha,\beta)}_k(x)w_{\alpha,\beta}(x)\,dx = 0
\quad\text{if}\quad k\ne n.
\]
The usual normalization supposes that
$P_{n}^{(\alpha,\beta)}(1)=\binom {n+\alpha}{ n}=\frac{(\alpha+1)_{n}}{n!}$, where we applied
the so-called \emph{Pochhammer symbol} or the \emph{rising factorial} defined as
\[(\alpha+1)_n \colonequals (\alpha+1)\cdot(\alpha+2)\cdots(\alpha+n).\]
In this notation we have
\begin{equation}\label{eq:def_Pn}
    \ P_{n}^{(\alpha,\beta)}(x)
    =\frac{(\alpha+1)_n}{n!}
        \sum_{k=0}^{\infty}\frac{(-n)_k(n+\alpha+\beta+1)_k}{k!\, (\alpha+1)_k}
        \left(\frac{1-x}{2}\right)^k,
    \quad
    n=1,2,\dots.
\end{equation}

\begin{definition} [{see \cite[p.~17]{Ch}}]

    We say that the zeros of the polynomials $g(x)$ and $h(x)$ \emph{interlace} (or
    \emph{interlace strictly}) if the following conditions hold simultaneously:
    \begin{compactitem}
    \item all zeros of~$g(x)$ and~$h(x)$ are simple, real and distinct (i.e. the polynomials are
        coprime),
    \item between each two consecutive zeros of~$g(x)$ there is exactly one zero of the
        polynomial~$h(x)$, and
    \item between each two consecutive zeros of~$h(x)$ there is exactly one zero of the
        polynomial~$g(x)$.
    \end{compactitem}\vspace{4pt}
    
    We say that the zeros of the polynomials~$g(x)$ and~$h(x)$ \emph{interlace non-strictly} if
    their zeros are real and become strictly interlacing after dividing both polynomials by the
    greatest common divisor $\gcd(g,h)$. Roughly speaking, the zeros of two polynomials
    interlace non-strictly if they can meet but never pass through each other when changing
    continuously from a strictly interlacing state.
\end{definition}
\begin{definition}{\refstepcounter{dummy}\label{def:real_pair}}
    A pair $\left(g(x),h(x)\right)$ is called \emph{real} if for any real numbers~$A,B$ the
    combination~$Ag(x)+Bh(x)$ has only real zeros. This is equivalent to the non-strict
    interlacing property of~$g(x)$ and~$h(x)$, which is shown in e.g.~\cite[Chapter~I]{Ch}.
\end{definition}
\begin{rem}
    The phrases \emph{``$g(x)$ and $h(x)$ interlace''},
    \emph{``$g(x)$ and $h(x)$ possess the interlacing property''},
    \emph{``$g(x)$ interlaces $h(x)$\/''},
    \emph{``$g(x)$ and $h(x)$ have interlacing zeros''} and
    \emph{``the zeros of $g(x)$ and $h(x)$ are interlacing''}
    we use synonymously.
\end{rem}

It is well-known that the orthogonal polynomials on the real line have real interlacing zeros
(due to the so-called three-term recurrence; see e.g.~\cite[pp.~42--47, Sections~3.2--3.3]{Sz}).
That is, in particular, the zeros of $P_n^{(\alpha,\beta)}$ and $P_{n-1}^{(\alpha,\beta)}$
interlace for all natural~$n$. In the present note we study zeros of polynomials that do not
satisfy the three-term recurrence. More specifically, we consider
\begin{equation}\label{eq:phi_expl}
    \phi_{n}^{(\alpha,\beta)}(\mu)
    \colonequals\sum_{k=0}^{\left[ n/2 \right]}
    \left. \frac{d^{2k}}{dx^{2k}} P_{n}^{(\alpha,\beta)}(x) \right|_{x=1} \cdot \mu^k
    =\frac{(\alpha+1)_n}{n!}
    \sum_{k=0}^{\left[ n/2 \right]}
    \frac{(-n)_{2k}(n+\alpha+\beta+1)_{2k}}{(\alpha+1)_{2k}}
    \left(\frac{\mu}{4}\right)^k,
    \quad
    n=1,2,\dots,
\end{equation}
where the notation~$[a]$ stands for the integer part of the number~$a$.
\begin{theorem*}[Theorem CCW \textnormal{\sffamily(Csordas, Charalambides and Waleffe~\cite{Cs})}]
    \refstepcounter{dummy}\namedlabel{th:Cs}{CCW}
    For every positive integer $n \geqslant 2$ the polynomial $\phi_n^{(\alpha,\beta)}(\mu)$,
    $-1 <\alpha < 1$, $-1<\beta$, has only real negative zeros.
\end{theorem*}
The proof of this theorem given in~\cite{Cs} rests on the Hermite-Biehler theory (see
Theorem~\ref{th:H-B} herein).

\begin{rem}[to Theorem~\ref{th:Cs}]
    In fact, the authors have shown that $\phi_n^{(\alpha,\beta)}(\mu)$ interlaces~$\phi_{n-1}^{(\alpha+1,\beta+1)}(\mu)$. As a result,
    \emph{the theorem remains valid when $1\leqslant\alpha<2$ and~$0<\beta$}. Note that the
    interlacing property here is strict, so it implies the simplicity of the zeros. Furthermore,
    \emph{the polynomial $\phi_n^{(\alpha,\beta)}(\mu)$ has only simple negative zeros for
        $-1<\alpha<0$ and~$-2<\beta$,}
    as well. This follows as a straightforward consequence of Theorems~\ref{th:CWst} and
    Lemma~\ref{lem:main} of the present study.
\end{rem}

Based on Theorem~\ref{th:Cs} the authors of~\cite{Cs} conjectured that these polynomials also
have the following property.
\begin{theorem*}[Conjecture A \textnormal{\sffamily({\cite[p.~3559]{Cs}})}]
    \refstepcounter{dummy}\namedlabel{con:main}{A}
    For $-1<\alpha<1$, $-1<\beta$ and $n \geqslant 4$ the zeros of the polynomials
    $\phi_{n}^{(\alpha,\beta)}$ and $\phi_{n-1}^{(\alpha,\beta)}$ interlace.
\end{theorem*}
In particular, this assertion would imply that the spectra of polynomial approximations to the
corresponding differential operator are negative and simple (see~\cite{CW}). In~\cite{CW},
Conjecture~\ref{con:main} was proved for $-1<\alpha,\beta<0$ and $0<\alpha,\beta<1$: see
Theorem~\ref{th:CW} below. In fact, the upper bound on~$\beta$ is redundant.
Theorem~\ref{th:CWst} with a shorter proof states that the conjecture holds true for
\begin{equation}\label{eq:cond_conjA}
    -1<\alpha<0,\ -1<\beta\quad\text{or}\quad
    0\leqslant\alpha<1,\ 0<\beta\quad\text{or}\quad
    1\leqslant\alpha<2,\ 1<\beta.
\end{equation}

Additionally, we study another assertion about the same polynomials.
\begin{theorem*}[Conjecture B \textnormal{\sffamily({\cite[p.~3559]{Cs}})}]
    \refstepcounter{dummy}\namedlabel{con:main_2}{B}
    For $\phi_n^{(\alpha,\beta)}(\mu)$ as in Theorem \ref{th:Cs} and for all $n\geqslant 5$, the zeros
    of the polynomials $\phi_n^{(\alpha,\beta)}(\mu)$ and $\phi_{n-2}^{(\alpha,\beta)}(\mu)$
    interlace.
\end{theorem*}
\setstretch{1.15}
Originally, this conjecture was stated for $-1<\alpha<1$ and~$\beta>-1$. However, numerical
calculations show that it fails for some values satisfying $-1<\beta<0<\alpha<1$. Our (partial)
solution to Conjecture~\ref{con:main_2} is given in Theorem~\ref{thm:cb_main_2}: it holds true
for~\(-1<\alpha<0<\beta\) \ or~~\(0<\alpha<1<\beta\). We approach by extending the idea
of~\cite{CW} to another pair of auxiliary polynomials. Certainly, there exists a relation
between Conjecture~\ref{con:main} and Conjecture~\ref{con:main_2} as discussed in
Section~\ref{sec:concl-relat-conj}.

Vieta's formulae imply that the sum of all zeros of~$\phi_n^{(\alpha,\beta)}(\mu)$ tends to
$-1/2$ for even~$n$ and to~$-1/6$ for odd~$n$ as~$n\to\infty$. Thus, the assertion of
Conjecture~\ref{con:main_2} gives that the zero points of $\phi_n^{(\alpha,\beta)}(\mu)$
converge monotonically in~$n$ outside of any fixed interval containing the origin. If the
assertions of both conjectures hold, then the
fraction~$\phi_{2n-1}^{(\alpha,\beta)}(\mu)/\phi_{2n}^{(\alpha,\beta)}(\mu)$ maps the upper half
of the complex plane into itself and converges to a function meromorphic outside of any disk
centred at the origin. This situation resembles how the quotients of orthogonal polynomials of
the first and second kinds behave.

Section~\ref{sec:some-polyn-relat} of the present paper introduces connections between
polynomials~$\phi_n^{(\alpha,\beta)}(\mu)$ with different~$n$, $\alpha$ and~$\beta$. These
connections allow us to extend and clarify the result~\cite{CW} in
Section~\ref{sec:results-conj-refc} (see Theorem~\ref{th:CWst}). We show that
Conjecture~\ref{con:main} holds true under the conditions~\eqref{eq:cond_conjA}.
Section~\ref{sec:results-conj-refc-B} contains the proof of Conjecture~\ref{con:main_2} for
\(-1<\alpha<0<\beta\) and \(0<\alpha<1<\beta\) (see Theorem~\ref{thm:cb_main_2}). In the last
section we show that the studied conjectures are actually related.

\section{Basic relations between the polynomials~\texorpdfstring{$\phi_n^{(\alpha,\beta)}$}{\textphi\_n(\textalpha,\textbeta)} for various \texorpdfstring{$\alpha$}{\textalpha}
    and~\texorpdfstring{$\beta$}{\textbeta}}\label{sec:some-polyn-relat}
Being connected with the Jacobi polynomials, the family
$\left(\phi_n^{(\alpha,\beta)}\right)_n$, where $n=2,3,\dots$, inherits some of their
properties. The formulae induced by the corresponding relations for the Jacobi case include (we
omit the argument $\mu$ of $\phi_n^{(\alpha,\beta)}$ for brevity's sake):
\begin{align}
    (2n+\alpha+\beta) \phi_{n}^{(\alpha,\beta-1)}
    &=(n+\alpha+\beta)\phi_n^{(\alpha,\beta)}+(n+\alpha)\phi_{n-1}^{(\alpha,\beta)}\label{eq:fl_summ1}
    ,\\
    (2n+\alpha+\beta) \phi_{n}^{(\alpha-1,\beta)}
    &=(n+\alpha+\beta)\phi_{n}^{(\alpha,\beta)}-(n+\beta)\phi_{n-1}^{(\alpha,\beta)}\label{eq:fl_summ2}
    ,\\
    (n+\alpha+\beta)\phi_{n}^{(\alpha,\beta)}
    &\xlongequal{\eqref{eq:fl_summ1}+\eqref{eq:fl_summ2}}(n+\beta) \phi_{n}^{(\alpha,\beta-1)}+(n+\alpha)\phi_{n}^{(\alpha-1,\beta)}\label{eq:fl_summ3}
    ,\\
    \phi_{n-1}^{(\alpha,\beta)}
    &\xlongequal{\eqref{eq:fl_summ1}-\eqref{eq:fl_summ2}}
    \phi_{n}^{(\alpha,\beta-1)}-\phi_{n}^{(\alpha-1,\beta)}\label{eq:fl_summ4}
    .
\end{align}
The latter two identities contain labels of equations above the equality sign: we use the
convention that this explains how the equalities can be obtained (up to some proper
coefficients). For example,~$2n+\alpha+\beta$ times~\eqref{eq:fl_summ3} is the sum of~$n+\beta$
times~\eqref{eq:fl_summ1} and~$n+\alpha$ times~\eqref{eq:fl_summ2}, whereas~$2n+\alpha+\beta$
times~\eqref{eq:fl_summ4} is the difference~$\eqref{eq:fl_summ1}-\eqref{eq:fl_summ2}$. The
identities~\eqref{eq:fl_summ1} and~\eqref{eq:fl_summ2} can be checked by applying the formulae
(see, e.g.,~\cite[p.~737]{Pr})
\begin{align}
    (2n+\alpha+\beta) P_{n}^{(\alpha,\beta-1)}(x)
    &=(n+\alpha+\beta)P_n^{(\alpha,\beta)}(x)+(n+\alpha)P_{n-1}^{(\alpha,\beta)}(x)
    \label{eq:fl_P_summ1}
    ,\\
    (2n+\alpha+\beta) P_{n}^{(\alpha-1,\beta)}(x)
    &=(n+\alpha+\beta)P_{n}^{(\alpha,\beta)}(x)-(n+\beta)P_{n-1}^{(\alpha,\beta)}(x)
    \label{eq:fl_P_summ2}
    ,
\end{align}
respectively, to the left-hand side of the definition~\eqref{eq:phi_expl}. The relations
involving derivatives (not surprisingly) differ from those for the Jacobi polynomials:
\begin{align}
    (n+\alpha)\phi_{n-1}^{(\alpha,\beta+1)}
    &=n \phi_n^{(\alpha,\beta)} - 2\mu \left(\phi_n^{(\alpha,\beta)}\right)'\label{eq:fl_diff1}
    ,\\
    (n+\alpha+\beta+1)\phi_{n}^{(\alpha,\beta+1)}
    &\xlongequal{\eqref{eq:fl_diff1}\text{ and }\eqref{eq:fl_summ1}}
    (n+\alpha+\beta+1)\phi_n^{(\alpha,\beta)} + 2\mu \left(\phi_n^{(\alpha,\beta)}\right)'\label{eq:fl_diff2}
    ,\\
    (n+\alpha)\phi_{n}^{(\alpha-1,\beta+1)}
    &\xlongequal{\eqref{eq:fl_summ2},\text{ then }\eqref{eq:fl_diff2}-\eqref{eq:fl_diff1}}
    \alpha \phi_n^{(\alpha,\beta)} + 2\mu \left(\phi_n^{(\alpha,\beta)}\right)'\label{eq:fl_diff3}
    ,\\
    \frac{n+\beta}{n+\alpha+\beta} 2\mu\left(\phi_{n}^{(\alpha,\beta-1)}\right)'
    &\xlongequal{\eqref{eq:fl_diff2}\text{ and }\eqref{eq:fl_summ3}}
    (n+\alpha)\phi_{n}^{(\alpha-1,\beta)}-\alpha\phi_{n}^{(\alpha,\beta)}\label{eq:fl_diff4}
    ,\\
    2\mu \left( \phi_n^{(\alpha,\beta-1)}\right)'
    &\xlongequal{\eqref{eq:fl_diff1}\text{ and }\eqref{eq:fl_summ4}}
    n\phi_n^{(\alpha-1,\beta)}-\alpha\phi_{n-1}^{(\alpha,\beta)}\label{eq:fl_diff5}
    .
\end{align}
So we see that the presented identities are not independent in the sense that they can be
obtained as combinations of others with various $\alpha$ and $\beta$. We derive the
formula~\eqref{eq:fl_diff1} with the help of the right-hand side of~\eqref{eq:phi_expl}.
\[
\begin{aligned}
    n \phi_n^{(\alpha,\beta)} - 2\mu \left(\phi_n^{(\alpha,\beta)}\right)'
    ={}&\frac{(\alpha+1)_n}{n!} \sum_{k=0}^{\left[ n/2 \right]}
        \frac{(-n)_{2k}(n+\alpha+\beta+1)_{2k}}{(\alpha+1)_{2k}}
        \left(\frac{\mu}{4}\right)^k \left(n-2k\right)\\
    ={}&(n+\alpha)\frac{(\alpha+1)_{n-1}}{(n-1)!} \sum_{k=0}^{\left[ n/2 \right]}
        (-n+2k)\frac{(-n)(-n+1)\cdots(-n+2k-1)(n+\alpha+\beta+1)_{2k}}{(-n)(\alpha+1)_{2k}}
        \left(\frac{\mu}{4}\right)^k\\
    ={}&(n+\alpha)\phi_{n-1}^{(\alpha,\beta+1)}.
\end{aligned}
\]
Certainly, we can combine formulae further obtaining e.g.
\begin{gather}    
    n \phi_{n}^{(\alpha,\beta)}-2\mu\left(\phi_n^{(\alpha,\beta)}\right)'
    \xlongequal{\eqref{eq:fl_diff1}} (n+\alpha)\phi_{n-1}^{(\alpha,\beta+1)}
    \xlongequal{\eqref{eq:fl_diff2}} (n+\alpha)\phi_{n-1}^{(\alpha,\beta)}
      +\frac{n+\alpha}{n+\alpha+\beta}2\mu \left(\phi_{n-1}^{(\alpha,\beta)}\right)'\label{eq:chain_rel_1}
    ,\\
    n\phi_n^{(\alpha,\beta)}-(n+\alpha)\phi_{n-1}^{(\alpha,\beta)}
    \xlongequal{\eqref{eq:fl_diff2}-\eqref{eq:fl_diff1}}\
    \frac{2n+\alpha+\beta}{n+\alpha+\beta}2\mu\left( \phi_n^{(\alpha,\beta-1)}\right)'
    \xlongequal{\frac{d}{d\mu}\eqref{eq:fl_summ1}}2\mu
    \left(\phi_n^{(\alpha,\beta)}\right)'
      +\frac{n+\alpha}{n+\alpha+\beta}2\mu\left(\phi_{n-1}^{(\alpha,\beta)} \right)'
    .\label{eq:chain_rel_2}
\end{gather}
The key role further plays the following combination of~\eqref{eq:fl_diff1}
and~\eqref{eq:fl_diff2}, which is valid for an arbitrary real~$A$,
\begin{equation} \label{eq:main_rel}
    \frac{n+\alpha+\beta}{n+\alpha}\phi_n^{(\alpha,\beta)}+A \phi_{n-1}^{(\alpha,\beta)}
    =\frac{(1+A)n+\alpha+\beta}{n+\alpha} \phi_n^{(\alpha,\beta-1)}
      +2\mu\frac{1-A}{n+\alpha}\left( \phi_n^{(\alpha,\beta-1)}\right)'.
\end{equation}
The next identity stands apart and can be checked explicitly with the help
of~\eqref{eq:phi_expl}
\begin{equation*}
    \phi_n^{(\alpha,\beta)}(\mu)-\phi_n^{(\alpha,\beta)}(0)
    =\tfrac 14 (n+\alpha+\beta+1)_2 \cdot\mu \phi_{n-2}^{(\alpha+2,\beta+2)}(\mu).
\end{equation*}
It reflects the standard formula for the derivative of the Jacobi polynomial
(e.g.~\cite[p.~737]{Pr}):
\begin{equation} \label{eq:diff_P}
\left(P_n^{(\alpha,\beta)}(x)\right)^{(m)}=2^{-m} (n+\alpha+\beta+1)_m\,P_{n-m}^{(\alpha+m,\beta+m)}(x).
\end{equation}

\begin{rem}
    Note that the equalities~\eqref{eq:fl_summ1}--\eqref{eq:diff_P} are of formal nature, and
    therefore their validity requires no orthogonality from the Jacobi polynomials. That is,
    these equalities holds true if all coefficients are defined, not only for $\alpha,\beta>-1$.
\end{rem}
\begin{rem}
    A polynomial with only real zeros interlaces its derivative by Rolle's theorem.%
    \footnote{Non-strictly whenever the polynomial has a multiple zero.} Consequently, they both
    interlace any real combination of them.%
    \footnote{See the definition of a real pair on the page~\pageref{def:real_pair}.} So if in
    one of the formulae \eqref{eq:fl_diff1}--\eqref{eq:fl_diff3} the first term on the
    right-hand side has only real zeros, then all three involved polynomials are pairwise
    interlacing. For example, if $\phi_n^{(\alpha,\beta)}$ has only real zeros,
    then~$\phi_{n-1}^{(\alpha,\beta+1)}$, $\phi_n^{(\alpha,\beta)}$ and
    $2\mu \left(\phi_n^{(\alpha,\beta)}\right)'$ are pairwise interlacing which is provided
    by~\eqref{eq:fl_diff1}.
\end{rem}
\begin{rem}
    The identities~\eqref{eq:fl_summ3} and \eqref{eq:fl_summ4} show that the interlacing
    property of~$\phi_n^{(\alpha,\beta)}(\mu)$ and~$\phi_{n-1}^{(\alpha,\beta)}(\mu)$ can also
    be expressed as the interlacing property of~$\phi_n^{(\alpha,\beta-1)}(\mu)$
    and~$\phi_{n}^{(\alpha-1,\beta)}(\mu)$. Analogously, from the relations~\eqref{eq:fl_diff4}
    and~\eqref{eq:fl_diff5} it can be seen that this is also equivalent to the interlacing
    property of~$\left(\phi_{n}^{(\alpha,\beta-1)}(\mu)\right)'$
    and~$\phi_{n}^{(\alpha-1,\beta)}(\mu)$.
\end{rem}
\begin{lemma}
    If $-1<\alpha<1$, $-1<\beta$ or if~$1\leqslant\alpha<2$, $0<\beta$, then the polynomials
    $\left(\phi_n^{(\alpha,\beta)}(\mu)\right)'$, \
    $\left(\phi_{n-1}^{(\alpha,\beta)}(\mu)\right)'$ and
    $\left(\phi_{n}^{(\alpha,\beta-1)}(\mu)\right)'$ are pairwise interlacing.
\end{lemma}
\begin{proof}
    By Theorem~\ref{th:Cs} the polynomials~$\phi_n^{(\alpha,\beta)}(\mu)$,
    $\phi_{n-1}^{(\alpha,\beta)}(\mu)$ have only (simple) negative zeros. Then the
    relation~\eqref{eq:fl_diff1} shows that the polynomial~$\phi_{n-1}^{(\alpha,\beta+1)}(\mu)$
    with negative zeros interlaces (strictly) both~$\phi_n^{(\alpha,\beta)}(\mu)$
    and~$\mu\left(\phi_n^{(\alpha,\beta)}(\mu)\right)'$.  Moreover, the sign
    of~$\phi_{n-1}^{(\alpha,\beta+1)}(\mu)$ at the origin and at the rightmost zero
    of~$\phi_n^{(\alpha,\beta)}(\mu)$ is the same, and
    therefore~$\left(\phi_n^{(\alpha,\beta)}(\mu),\mu\phi_{n-1}^{(\alpha,\beta+1)}(\mu)\right)$
    is a real coprime pair. A similar consideration of the relation~\eqref{eq:fl_diff2} gives
    that the pair
    $ \left(\mu\phi_{n-1}^{(\alpha,\beta)}(\mu),\phi_{n-1}^{(\alpha,\beta+1)}(\mu)\right) $ is
    also real and coprime.

    In particular, each interval between two consequent zeros
    of~$\phi_{n-1}^{(\alpha,\beta+1)}(\mu)$ contains exactly one zero
    of~$\phi_n^{(\alpha,\beta)}(\mu)$ as well as of~$\phi_{n-1}^{(\alpha,\beta)}(\mu)$. Taking
    the signs of the last two polynomials at the ends of these intervals into account shows that
    the difference in the left-hand side of~\eqref{eq:chain_rel_2}, and
    hence~$\left(\phi_{n}^{(\alpha,\beta-1)}(\mu)\right)'$, changes its sign between consecutive
    zeros of~$\phi_{n-1}^{(\alpha,\beta+1)}(\mu)$. Since
    $\deg \left(\phi_{n}^{(\alpha,\beta-1)}\right)' \leqslant \deg\phi_{n-1}^{(\alpha,\beta+1)}$, the
    polynomial $\mu\left(\phi_{n}^{(\alpha,\beta-1)}(\mu)\right)'$ necessarily
    interlaces~$\phi_{n-1}^{(\alpha,\beta+1)}(\mu)$. Then the right-hand side
    of~\eqref{eq:chain_rel_2} shows that
    \begin{equation}\label{eq:apdx1}
        \left(\phi_n^{(\alpha,\beta)}(\mu)\right)'
        +\frac{n+\alpha}{n+\alpha+\beta}\left(\phi_{n-1}^{(\alpha,\beta)}(\mu) \right)'
    \end{equation}
    and~$\phi_{n-1}^{(\alpha,\beta+1)}(\mu)$ have interlacing zeros. At the same time, by
    differentiating the equality~\eqref{eq:chain_rel_2} we obtain that
    \begin{equation}\label{eq:apdx2}
        n\left(\phi_n^{(\alpha,\beta)}(\mu)\right)'
        -(n+\alpha)\left(\phi_{n-1}^{(\alpha,\beta)}(\mu) \right)'
    \end{equation}
    is proportional to~$\left(\mu\left(\phi_{n}^{(\alpha,\beta-1)}(\mu)\right)'\right)'$ and,
    hence, interlaces~$\mu\left(\phi_{n}^{(\alpha,\beta-1)}(\mu)\right)'$. Put in other words,
    the polynomials~\eqref{eq:apdx1} and~\eqref{eq:apdx2} are interlacing. With appropriate
    factors, their sum gives~\(\left(\phi_n^{(\alpha,\beta)}(\mu)\right)'\) and their difference
    gives~\(\left(\phi_{n-1}^{(\alpha,\beta)}(\mu) \right)'\). This yields the lemma.
\end{proof}

\section{Results on Conjecture~\texorpdfstring{\ref{con:main}}{A}}\label{sec:results-conj-refc}
\setstretch{1.2}
\begin{theorem*}[Theorem HB
    \textnormal{\sffamily(Hermite-Biehler, for real polynomials)}]
    \refstepcounter{dummy}\namedlabel{th:H-B}{HB}
    A real polynomial \(f (z) \colonequals p(z^2) + zq(z^2)\) is stable%
    \footnote{That is, Hurwitz stable: $f(z)=0\implies \Re z < 0$.} if and only if
    $p(0)\cdot q(0)>0$ and all zeros of~$p(z)$ and~$zq(z)$ are nonpositive and strictly
    interlacing.
\end{theorem*}
This well-known fact plays a crucial role in~\cite{Cs} for proving Theorem~\ref{th:Cs}. Here the
statement of the Hermite-Biehler theorem is expressed closely to the one given
in~\cite[p.~228]{Ga}.
The replacement of~$q(z)$ with~$zq(z)$ provides the desired order of zeros. That is, the zero
of~$p(z)$ and~$q(z)$ closest to the origin belongs to the former polynomial. The more general
statement can be found in e.g.~\cite[p.~21]{Ch}. The present study also uses
Theorem~\ref{th:H-B} as a main tool.

Some bounds on the parameters~$\alpha$ and~$\beta$ are necessary even for Theorem~\ref{th:Cs}
(i.e. are satisfied if all zeros of~$\phi_n^{(\alpha,\beta)}(\mu)$ are negative for
every~$n\geqslant 2$). The restriction~$\alpha>-1$ corresponds to positivity of the coefficients (and
to negativity of all zeros). The parameter~$\alpha$ is bounded from above by~$3.37228\dots$.
Indeed, if the polynomial $\phi_n^{(\alpha,\beta)}(\mu)\equalscolon\sum_{k=0}^{m}b_k\mu^k$,
$m=[n/2]$, has only negative%
\footnote{Here we suppose that~$\phi_n^{(\alpha,\beta)}(\mu)$ has only simple zeros; the case of
    multiple zeros follows by continuity.}
zeros, then by Rolle's theorem, the zeros
of~$p(\mu)\colonequals\mu^m\phi_n^{(\alpha,\beta)}(\mu^{-1})$ and~$p'(\mu)$ are interlacing. So,
Theorem~\ref{th:H-B} then implies that the polynomial $p(\mu^2)+\mu p'(\mu^2)$ is stable.
Therefore, when~$b_0>0$ the Hurwitz criterion (see e.g.~\cite[Chapter~I]{Ch}) gives us the
easy-to-check conditions $b_1>0$ and
\begin{equation}\label{eq:cond_stab}
    \begin{vmatrix}
       mb_0&(m-1)b_1&(m-2)b_2\\
       b_0 & b_1&     b_2\\
       0 & mb_0&(m-1)b_1
     \end{vmatrix}\geqslant 0,
     \quad\text{that is}\quad
     \frac{b_1^2}{b_0b_2}\geqslant\frac{2 m}{m-1}>2,
\end{equation}
which are necessarily true when the polynomial $p(\mu)$ has only real zeros. In our case we have
\[
\frac{b_1^2}{b_0b_2}
=\frac{n(n-1)(\alpha+3)_2(n+\alpha+\beta+1)_2}
      {(n-2)(n-3)(\alpha+1)_2(n+\alpha+\beta+3)_2}
      \xrightarrow{n\to\infty}\frac{(\alpha+3)(\alpha+4)}{(\alpha+1)(\alpha+2)},
\]
so the condition~\eqref{eq:cond_stab} fails to be true (along with the assertion of
Theorem~\ref{th:Cs}) for every~$n$ big enough and every~$\beta$
when~$\alpha>\frac{1+\sqrt{33}}2=3.37228\dots$ or $\alpha<\frac{1-\sqrt{33}}2=-2.37228\dots$.
Computer experiments show that the polynomials~$\phi_n^{(\alpha,\beta)}$ with
\emph{positive coefficients} can have zeros outside the real axis for a (large enough)
negative~$\beta$. So, some lower constraint on the parameter~$\beta$ is also required.

\begin{definition}
    Denote the $i$th zero of a polynomial~$p$ with respect to the distance from the origin
    by~$\zr_i(p)$. Put~$\zr_i(p)$ equal to~$-\infty$ if~$\deg p<i$ and to zero if~$i=0$ (it is
    convenient since all coefficients of the polynomials we deal with are nonnegative).
\end{definition}
\begin{lemma} \label{lem:main} Let $\alpha> -1$ and $n+\alpha+\beta> 0$, $n=4,5,\dots$. The
    zeros of the polynomials $\phi_{n}^{(\alpha,\beta)}$ and $\phi_{n-1}^{(\alpha,\beta)}$ are
    negative and interlace non-strictly (strictly) if and only if the
    polynomial~$\phi_{n}^{(\alpha,\beta-1)}$ has only real zeros (only simple real zeros,
    respectively). Moreover, if~$\phi_{n}^{(\alpha,\beta-1)}$ has only real zeros, then
    $\zr_1\left(\phi_{n-1}^{(\alpha,\beta)}\right)\leqslant\zr_1\left(\phi_{n}^{(\alpha,\beta)}\right)$.
\end{lemma}
\begin{proof}
    The relation~\eqref{eq:main_rel} with~$A=1$ and~$A=-(n+\alpha+\beta)/n$ implies that each
    common zero of the polynomials~$\phi_{n-1}^{(\alpha,\beta)}$ and~$\phi_{n}^{(\alpha,\beta)}$
    is a multiple zero of~$\phi_{n}^{(\alpha,\beta-1)}$. The converse result is given
    by~\eqref{eq:fl_diff1} and~\eqref{eq:fl_diff2}.
    
    Suppose that~$\phi_{n}^{(\alpha,\beta-1)}$ has only real zeros. The coefficients of this
    polynomial are positive under the assumptions of the lemma, and hence all of its zeros are
    negative. By Rolle's theorem, the
    pair~$\left(\phi_{n}^{(\alpha,\beta-1)},\mu\left(\phi_{n}^{(\alpha,\beta-1)}\right)'\right)$
    is real. Therefore, the polynomials $\phi_{n-1}^{(\alpha,\beta)}$ and
    $\phi_{n}^{(\alpha,\beta)}$ have only real zeros by the formulae~\eqref{eq:fl_diff1}
    and~\eqref{eq:fl_diff2}, respectively. The zeros are negative automatically since the
    coefficients of polynomials are positive. Moreover, we have that the first zero
    of~$\phi_{n}^{(\alpha,\beta)}$ is closer to the origin than that
    of~$\phi_{n-1}^{(\alpha,\beta)}$. The relation~\eqref{eq:main_rel} holds for all real~$A$,
    which yields that the polynomials~$\phi_{n-1}^{(\alpha,\beta)}$
    and~$\phi_{n}^{(\alpha,\beta)}$ form a real pair and thus have (non-strictly) interlacing
    zeros.

    Let~$\phi_{n-1}^{(\alpha,\beta)}$ and $\phi_{n}^{(\alpha,\beta)}$ have negative interlacing
    zeros. Then any of their real combinations only has real zeros. This is true
    for~$\phi_{n}^{(\alpha,\beta-1)}$ according to the identity~\eqref{eq:fl_summ1}.
\end{proof}
\begin{corollary}\label{cor:main}
    For $-1<\alpha<1$, $\beta>0$ or $1\leqslant\alpha<2$, $\beta>1$ the zeros of the
    polynomials~$\phi_{n}^{(\alpha,\beta)}$ and $\phi_{n-1}^{(\alpha,\beta)}$, \ $n = 4,5,\dots$,
    interlace. (Furthermore, the polynomials~$\phi_{n}^{(\alpha,\beta)}$ and
    $\mu\phi_{n-1}^{(\alpha,\beta)}$ interlace.)
\end{corollary}
\begin{proof}
    This corollary is provided by Theorem~\ref{th:Cs} (see also the remark on it) and
    Lemma~\ref{lem:main}.
\end{proof}

\begin{theorem*}[Theorem CW \textnormal{\sffamily(Theorem~3.10, Charalambides et al.~{\cite{CW}})}]
    \refstepcounter{dummy}\namedlabel{th:CW}{CW}
    Conjecture~\ref{con:main} holds true for~$-1<\alpha,\beta<0$ and for~$0<\alpha,\beta<1$.
\end{theorem*}
This theorem relies on Theorem~3.8 and Theorem~3.9 of the same work and on Theorem~\ref{th:Cs}.
In fact, the original proof (which we extend in the next section to treat
Conjecture~\ref{con:main_2}) does not need any upper bound on~$\beta$. It becomes more evident
on recalling that the region of positive~$\beta$ is covered by Corollary~\ref{cor:main} (i.e. is
a straightforward consequence of Lemma~\ref{lem:main} and Theorem~\ref{th:Cs}).
\begin{theorem*}[Theorem CW*]
    \refstepcounter{dummy}\namedlabel{th:CWst}{CW*}
    Conjecture~\ref{con:main} holds true for
    \(-1<\alpha<0\), \(-1<\beta\) \ or \
    \(0\leqslant\alpha<1\), \(0<\beta\) \ or \
    \(1\leqslant\alpha<2\), \(1<\beta\).
\end{theorem*}
\begin{proof}
    Corollary~\ref{cor:main} suits the case of positive~$\alpha$. For the region with
    negative~$\alpha$ it is enough to prove only Theorem~\ref{th:4} (which is stated below) and
    Theorem~3.9 is not needed. Indeed, according to Theorem~\ref{th:4} and Theorem~\ref{th:H-B}
    we have that all zeros of~$\phi_{n}^{(\alpha,\beta-1)}$ are simple and real for all~$n$, so
    Lemma~\ref{lem:main} is applicable.
\end{proof}
\begin{theorem}[Equivalent to Theorem~3.8, Charalambides et al.~{\cite{CW}}]\label{th:4}
    If $-1<\alpha<0$, $-1<\beta$ and $n=4,5,\dots$, then the zeros of the polynomial
    $\Phi_{n}(1;\mu)$, where
    \[
    \Phi_{n}(x;\mu)
    \colonequals\sum_{k=0}^{n}
    \frac{d^{k}}{dx^{k}} P_{n}^{(\alpha,\beta-1)}(x) \mu^k,
    \]
    lie in the open left half of the complex plane.
\end{theorem}
\begin{rem}
    It is worth noting that this theorem cannot be extended to the full
    range~$-1<\beta<0<\alpha<1$. According to computer experiments, Conjecture~\ref{con:main}
    seems to hold in this range, while Theorem~\ref{th:4} fails, e.g., for~$n=12$
    when~$\beta=-0.8$ and~$\alpha\gtrapprox 0.97842$, or when~$\beta=-0.9$
    and~$\vphantom{\Big|}\alpha\gtrapprox 0.97140$. The reason is that proving Conjecture~\ref{con:main} only
    requires negative simple zeros of the polynomial~$\phi_n^{(\alpha,\beta-1)}(\mu)$,
    Theorem~\ref{th:4} nevertheless asserts additional properties
    of~$\phi_{n-1}^{(\alpha+1,\beta)}(\mu)$ as given by the Hermite-Biehler theorem.
\end{rem}
\setstretch{1.15}
\begin{proof}
    This proof is akin to~\cite[Theorem~3.8]{CW} but uses other relations for the Jacobi
    polynomials. The polynomial~$\Phi_{n}(x;\mu)$ satisfies the differential equation (here we
    consider~$\mu$ as a parameter)
    \[
    \Phi_n(x;\mu) = P_n^{(\alpha,\beta-1)}(x) + \mu \frac{d\Phi_n(x;\mu)}{dx}.
    \]
    Let $\Phi_n\colonequals\Phi_n(x;\mu)$ for brevity and let $\overline{\dfrac{d\Phi_n}{dx}}$
    denote a complex conjugate of $\dfrac{d\Phi_n}{dx}$. Multiplying by
    $\overline{\dfrac{d\Phi_n}{dx}} w_{\alpha,\beta+1}$, where
    $w_{\alpha,\beta+1}\colonequals w_{\alpha,\beta+1}(x)=(1-x)^\alpha(1+x)^{\beta+1}$, and
    integration over the interval~$(-1,1)$ gives us
    \[
    \int_{-1}^1\Phi_n\overline{\dfrac{d\Phi_n}{dx}} w_{\alpha,\beta+1}\,dx
    = \int_{-1}^1 P_n^{(\alpha,\beta-1)}\overline{\dfrac{d\Phi_n}{dx}} w_{\alpha,\beta+1}\,dx
      + \mu \int_{-1}^1\left|\dfrac{d\Phi_n}{dx}\right|^2w_{\alpha,\beta+1}\, dx.
    \]
    Select $\mu$ so that \( \Phi_n(1;\mu)=0 \). To estimate the real part of~$\mu$ we add the last
    equation to its complex conjugate and obtain
    \begin{equation}    \label{eq:CW_main}
        \int_{-1}^1\frac {d\left(|\Phi_n|^2\right)}{dx} w_{\alpha,\beta+1}\,dx
        = \int_{-1}^1 P_n^{(\alpha,\beta-1)}\cdot
                     2\Re\dfrac{d\Phi_n}{dx}\cdot w_{\alpha,\beta+1}\,dx
        + 2\Re \mu \int_{-1}^1\left|\dfrac{d\Phi_n}{dx}\right|^2w_{\alpha,\beta+1}\, dx.
    \end{equation}
    Since $w_{\alpha,\beta+1}$ increases on~$(-1,1)$ and
    $\lim_{x\to-1+}\Phi_nw_{\alpha,\beta+1}=\lim_{x\to1-}\Phi_nw_{\alpha,\beta+1}=0$, the
    left-hand side satisfies
    \[
    \int_{-1}^1\frac {d\left(|\Phi_n|^2\right)}{dx} w_{\alpha,\beta+1}\,dx
    =-\int_{-1}^1|\Phi_n|^2 w_{\alpha,\beta+1}'\,dx
    <0.
    \]
    Applying the relation~\eqref{eq:fl_P_summ1} to the polynomial~$P_{n}^{(\alpha,\beta-1)}$
    three times gives us
    \begin{multline*}
        P_{n}^{(\alpha,\beta-1)}
        =\frac{n+\alpha+\beta}{2n+\alpha+\beta}P_n^{(\alpha,\beta)}
          +\frac{n+\alpha}{2n+\alpha+\beta}P_{n-1}^{(\alpha,\beta)}
        =
        \frac{(n+\alpha+\beta)_2}{(2n+\alpha+\beta)_2}P_n^{(\alpha,\beta+1)}
        +\frac{(n+\alpha+\beta)(n+\alpha)}{(2n+\alpha+\beta)_2}P_{n-1}^{(\alpha,\beta+1)}\\
        +\frac{(n+\alpha)(n+\alpha+\beta)}{(2n+\alpha+\beta-1)_2}P_{n-1}^{(\alpha,\beta+1)}
        +\frac{(n+\alpha-1)_2}{(2n+\alpha+\beta-1)_2}P_{n-2}^{(\alpha,\beta+1)},
    \end{multline*}
    that is,
    \[
        P_{n}^{(\alpha,\beta-1)}
        = \frac{(n+\alpha+\beta)_2}{(2n+\alpha+\beta)_2}P_n^{(\alpha,\beta+1)}
            +\frac{2(n+\alpha)(n+\alpha+\beta)}
               {(2n+\alpha+\beta-1)(2n+\alpha+\beta+1)}
              P_{n-1}^{(\alpha,\beta+1)}
          +\frac{(n+\alpha-1)_2}{(2n+\alpha+\beta-1)_2}P_{n-2}^{(\alpha,\beta+1)}.
    \]
    By the definition of~$\Phi_n$ and the formula~\eqref{eq:diff_P},
    \[
    \Re\frac {d\Phi_n}{dx}
    = 2^{-1} (n+\alpha+\beta)P_{n-1}^{(\alpha+1,\beta)}
      + \Re \mu\cdot 2^{-2} (n+\alpha+\beta)_2P_{n-2}^{(\alpha+2,\beta+1)}
      + \big\{\text{polynomials of degree }<n-2\big\}. 
    \]
    The difference~$\eqref{eq:fl_P_summ1}-\eqref{eq:fl_P_summ2}$ induces the identity
    \(P_{n-1}^{(\alpha+1,\beta)} = P_{n-1}^{(\alpha,\beta+1)}+P_{n-2}^{(\alpha+1,\beta+1)}\), so
    we finally have
    \[
    \int_{-1}^1 P_n^{(\alpha,\beta-1)}\cdot2\Re\frac{d\Phi_n}{dx}\cdot w_{\alpha,\beta+1}\,dx
    = \eta+\zeta \Re \mu,
    \quad\text{where}\quad
    \eta,\zeta>0.
    \]
    Now the terms of the relation~\eqref{eq:CW_main} are estimated, and it yields~$0>\Re\mu$.
\end{proof}

\section{Results on Conjecture~\texorpdfstring{\ref{con:main_2}}{B}}\label{sec:results-conj-refc-B}
We just have shown that the result on the Conjecture~\ref{con:main} in~\cite{CW} can be obtained
in a shorter way if we consider polynomials with shifted parameter values (we
used~$\phi_{n}^{(\alpha,\beta-1)}$) instead of real combinations of the
polynomials~$\phi_n^{(\alpha,\beta)}$ and $\phi_{n-1}^{(\alpha,\beta)}$. At the same time, to
verify the Conjecture~\ref{con:main_2} we can combine both these ideas.

For any fixed $n>3$ consider the intermediary polynomial
    \[
      f(x;\mu)\colonequals
      \sum_{k=0}^n \mu^k\left(
            A \frac{d^k}{dx^k}P_n^{(\alpha,\beta)}(x)
          + \mu\frac{d^k}{dx^k}P_{n-1}^{(\alpha,\beta)}(x)
      \right).
    \]
\setstretch{1.1}
\begin{lemma}\label{lemma:f_stable}
    The polynomial $f(1;\mu)$ is Hurwitz-stable provided that~$-1<\alpha<1$ and~$\beta,A>0$.
\end{lemma}
\begin{proof}
    From the definition of~$f(x,\mu)$ the differential equation
    \[
      \mu\frac{d}{dx}f(x;\mu) + A P_n^{(\alpha,\beta)}(x)
          + \mu P_{n-1}^{(\alpha,\beta)}(x) = f(x;\mu)
    \]
    follows. Multiplication by $\overline{f(x;\mu)}w_{\alpha-1,\beta}(x)$ gives us
    \[
      \overline{f}\frac{df}{dx} w_{\alpha-1,\beta}
          + A P_n^{(\alpha,\beta)}(x) \frac {\overline{f}}\mu w_{\alpha-1,\beta}
          + P_{n-1}^{(\alpha,\beta)}(x)\overline{f}w_{\alpha-1,\beta}
      = \frac 1\mu |f|^2 w_{\alpha-1,\beta},
    \]
    where we put \( f\colonequals f(x;\mu)\) and
    \(w_{\alpha-1,\beta}\colonequals w_{\alpha-1,\beta}(x)=(1-x)^{\alpha-1}(1+x)^\beta\) for
    brevity. Adding to this equality its complex conjugate and integrating yields
    \begin{multline}\label{eq:int_rel}
      \int_{-1}^1\frac{d\left(|f|^2\right)}{dx} w_{\alpha-1,\beta}\,dx
          + A \int_{-1}^1 P_n^{(\alpha,\beta)}(x)
            \left(\frac {\overline{f}}\mu + \frac {f}{\overline\mu}\right)w_{\alpha-1,\beta}\,dx
          + \int_{-1}^1 P_{n-1}^{(\alpha,\beta)}(x)\left(\overline{f} + f \right) w_{\alpha-1,\beta}\,dx
          \\= \left(\frac 1\mu + \frac 1{\overline\mu}\right)
              \int_{-1}^1 |f|^2 w_{\alpha-1,\beta}\,dx.
    \end{multline}
    
    Take $\mu$ so that $f(1;\mu)=0$. Then the polynomial~$f(x;\mu)$ can be represented as
    \[
        f(x;\mu)=(1-x)\sum_{k=0}^{n-1} c_k P_k^{(\alpha,\beta)}(x)
    \]
    with some complex constants $c_k$ depending on~$\mu$ (generally speaking). Observe
    that~$c_{n-1}<0$: denoting the leading coefficient in~$x$ by~$\lc$, we obtain
    \[
    c_{n-1}
      =\frac{\lc\left(f(x;\mu)\right)}{\lc\left(-x P_{n-1}^{(\alpha,\beta)}(x)\right)}
      =-\frac{A \cdot \lc\left(P_n^{(\alpha,\beta)}(x)\right)}{\lc\left(P_{n-1}^{(\alpha,\beta)}(x)\right)}
      \xlongequal{\eqref{eq:def_Pn}}-A\frac{(n+\alpha+\beta+1)_n\cdot(n-1)!\,2^{n-1}}{n!\,2^n\cdot(n+\alpha+\beta)_{n-1}}
      =-A\frac{(2n+\alpha+\beta-1)_2}{2n(n+\alpha+\beta)}<0.
    \]
    Then we have
    \begin{equation}\label{eq:int_rel1}
        \begin{aligned}
            \int_{-1}^1 P_n^{(\alpha,\beta)}(x) f w_{\alpha-1,\beta}\,dx
            &= \int_{-1}^1 P_n^{(\alpha,\beta)}(x)
               \sum_{k=0}^{n-1} c_k P_k^{(\alpha,\beta)}(x) w_{\alpha,\beta}\,dx
            =0,\\
            \int_{-1}^1 P_{n-1}^{(\alpha,\beta)}(x) f w_{\alpha-1,\beta}\,dx
            &= c_{n-1}\int_{-1}^1
                \left(P_{n-1}^{(\alpha,\beta)}(x)\right)^2 w_{\alpha,\beta}\,dx
            <0
        \end{aligned}
    \end{equation}
    as a consequence of orthogonality of the Jacobi polynomials.
    Note that if $-1<\alpha<1$ and $\beta>0$, then
    \[
    w_{\alpha-1,\beta}'=\left(\beta(1-x)-(\alpha-1)(1+x)\right)\cdot w_{\alpha-2,\beta-1} >0
    \quad\text{and}\quad
    \lim_{x\to-1+} |f|^2 w_{\alpha-1,\beta} =\lim_{x\to1-} |f|^2 w_{\alpha-1,\beta}=0.
    \]
    Therefore, integrating by parts yields
    \begin{equation}\label{eq:int_rel3}
        \int_{-1}^1\frac{d\left(|f|^2\right)}{dx} w_{\alpha-1,\beta}\,dx
        = - \int_{-1}^1 \left(|f|^2\right) w_{\alpha-1,\beta}' \,dx < 0.
    \end{equation}
    By the formulae~\eqref{eq:int_rel1}--\eqref{eq:int_rel3}, the left-hand side of the
    equation~\eqref{eq:int_rel} is negative, thus necessarily~$\Re\mu<0$.
\end{proof}

Consider also another intermediary polynomial for a fixed~$n>3$, namely
    \[
      g(x;\mu)\colonequals
      \sum_{k=0}^n \mu^{k}\left(
            \mu \frac{d^k}{dx^k}P_n^{(\alpha,\beta)}(x)
          + A\frac{d^k}{dx^k}P_{n-1}^{(\alpha,\beta)}(x)
      \right).
    \]
\begin{lemma}\label{lemma:g_stable}
    The polynomial $g(1;\mu)$ is Hurwitz-stable provided that~$-1<\alpha<0$ and~$\beta,A>0$.
\end{lemma}
\begin{proof}
    This proof is analogous to the proofs of Theorem~\ref{th:4} and Lemma~\ref{lemma:f_stable}.
    From the definition of~$g(x,\mu)$ we have
    \[
      \mu\frac{dg(x;\mu)}{dx} + \mu P_n^{(\alpha,\beta)}(x)
          + A P_{n-1}^{(\alpha,\beta)}(x) = g(x;\mu)
    \]
    which gives us (we put $g\coloneqq g(x;\mu)$ and
    $g'\coloneqq\frac{dg(x;\mu)}{dx}$ for brevity's sake)
    \[
      \mu g' \overline{g'} w_{\alpha,\beta}
          + \mu P_n^{(\alpha,\beta)}(x) \overline{g'} w_{\alpha,\beta} 
          + A P_{n-1}^{(\alpha,\beta)}(x)\overline{g'} w_{\alpha,\beta}
      = g\overline{g'}w_{\alpha,\beta}
    \]
    after multiplication by $\overline{g'}w_{\alpha,\beta}$. 
    Adding to the equation its complex conjugate and integrating yields
    \begin{multline}\label{eq:int_grel}
      (\mu+\overline{\mu})\int_{-1}^1 |g'|^2 w_{\alpha,\beta}\,dx
          + \int_{-1}^1 P_n^{(\alpha,\beta)}(x)
            \left(\mu \overline{g'} + \overline{\mu}g'\right)w_{\alpha,\beta}\,dx
          + A\int_{-1}^1 P_{n-1}^{(\alpha,\beta)}(x)\left(\overline{g'} + g' \right) w_{\alpha,\beta}\,dx
          \\= \int_{-1}^1 \left(g\overline{g'}+\overline{g}g'\right)w_{\alpha,\beta}\,dx.
    \end{multline}
    Observe that~$g$ is a polynomial of degree~$n$ in~$x$ and its leading coefficient is given
    by~$P_n^{(\alpha,\beta)}(x)\cdot\mu$. Consequently, substituting the explicit expression
    for~$\lc\left(P_n^{(\alpha,\beta)}(x)\right)$ from the formula~\eqref{eq:def_Pn} gives
    \[
      \lc (g')
      =n\lc\left(P_n^{(\alpha,\beta)}(x)\right)\mu
      = \frac{(2n+\alpha+\beta-1)_2}{2(n+\alpha+\beta)}
        \cdot \frac{(n+\alpha+\beta)_{n-1}}{(n-1)!\,2^{n-1}}\mu
      = \frac{(2n+\alpha+\beta-1)_2}{2(n+\alpha+\beta)}
        \lc\left(P_{n-1}^{(\alpha,\beta)}(x)\right)\mu.
    \]
    This allows us to calculate the third summand on the left-hand side of~\eqref{eq:int_grel}:
    \begin{multline}\label{eq:int_grel1}
    \int_{-1}^1 P_{n-1}^{(\alpha,\beta)}(x)\left(\overline{g'} + g' \right) w_{\alpha,\beta}\,dx
    =\int_{-1}^1 P_{n-1}^{(\alpha,\beta)}(x)
        \frac{(2n+\alpha+\beta-1)_2}{2(n+\alpha+\beta)}
        \lc\left(P_{n-1}^{(\alpha,\beta)}(x)\right)
        \left(\overline{\mu} + \mu \right)x^{n-1}w_{\alpha,\beta}\,dx\\
    =\left(\overline{\mu} + \mu \right)
        \frac{(2n+\alpha+\beta-1)_2}{2(n+\alpha+\beta)}
        \int_{-1}^1 \left(P_{n-1}^{(\alpha,\beta)}(x)\right)^2
    w_{\alpha,\beta}\,dx.
    \end{multline}
    Additionally, we have
    \begin{equation}\label{eq:int_grel2}
        \int_{-1}^1 P_n^{(\alpha,\beta)}(x) g' w_{\alpha,\beta}\,dx
        =0.
    \end{equation}

    Take $\mu$ so that $g(1;\mu)=0$. Then
    \[
    w_{\alpha,\beta}'=\left(\beta(1-x)-\alpha(1+x)\right)\cdot w_{\alpha-1,\beta-1} >0
    \quad\text{and}\quad
    \lim_{x\to-1+} |g|^2 w_{\alpha,\beta} =\lim_{x\to1-} |g|^2 w_{\alpha,\beta}=0
    \]
    since $-1<\alpha<0$ and $\beta>0$. Integrating by parts we obtain
    \begin{equation}\label{eq:int_grel3}
        \int_{-1}^1\left(g\overline{g'}+\overline{g}g'\right) w_{\alpha,\beta}\,dx
        = - \int_{-1}^1 |g|^2 w'_{\alpha,\beta} \,dx < 0.
    \end{equation}
    Now let us bring together the relations~\eqref{eq:int_grel}--\eqref{eq:int_grel3}:
    \[
      (\mu+\overline{\mu})\int_{-1}^1 \left(|g'|^2
          + A\frac{(2n+\alpha+\beta-1)_2}{2(n+\alpha+\beta)} \left(P_{n-1}^{(\alpha,\beta)}(x)\right)^2
      \right)w_{\alpha,\beta}\,dx
      = - \int_{-1}^1 |g|^2 w'_{\alpha,\beta} \,dx,
      \quad\text{hence}\quad
      2\Re\mu=\mu+\overline{\mu}<0.
    \]
    That is, any zero~$\mu$ of the polynomial~$g(1;\mu)$ resides in the left half of the complex
    plane.
\end{proof}

\begin{corollary}\label{cr:main_2}
    For any positive~$A$, the zeros of the polynomials
    \begin{equation}\label{eq:r_comb_2}
        2A\phi_n^{(\alpha,\beta)}(\mu)+ \left(n+\alpha+\beta\right)\mu\phi_{n-2}^{(\alpha+1,\beta+1)}(\mu)
        \quad\text{and}\quad
        A\left(n+\alpha+\beta+1\right)\phi_{n-1}^{(\alpha+1,\beta+1)}(\mu) + 2\phi_{n-1}^{(\alpha,\beta)}(\mu)
    \end{equation}
    are interlacing provided that $-1<\alpha<1$ and~$\beta>0$. If in addition $-1<\alpha<0$, the
    zeros of the polynomials
    \begin{equation}\label{eq:r_comb_3}
        \left(n+\alpha+\beta+1\right)\mu \phi_{n-1}^{(\alpha+1,\beta+1)}(\mu)+ 2A\phi_{n-1}^{(\alpha,\beta)}(\mu)
        \quad\text{and}\quad
        2\phi_{n}^{(\alpha,\beta)}(\mu) + A\left(n+\alpha+\beta\right)\phi_{n-2}^{(\alpha+1,\beta+1)}(\mu)
    \end{equation}
    are also interlacing.
\end{corollary}
\begin{proof}
    To get the assertion we apply the relation~\eqref{eq:diff_P} to the even and odd parts of
    the polynomials~$f(x;\mu)$ and~$g(x;\mu)$. The even part of~$f(x;\mu)$ is
    \begin{align*}
        \frac{f(x;\mu)+f(x;-\mu)}{2}
        &=
        A \sum_{k=0}^{[n/2]} \mu^{2k} \frac{d^{2k}}{dx^{2k}}P_n^{(\alpha,\beta)}(x)
        + \mu^2\sum_{k=0}^{[n/2]} \mu^{2k} \frac{d^{2k+1}}{dx^{2k+1}}P_{n-1}^{(\alpha,\beta)}(x)
        \\
        &= A \sum_{k=0}^{[n/2]} \mu^{2k} \frac{d^{2k}}{dx^{2k}}P_n^{(\alpha,\beta)}(x)
         + \frac{n+\alpha+\beta}{2}
           \mu^2\sum_{k=0}^{[n/2]-1} \mu^{2k} \frac{d^{2k}}{dx^{2k}}P_{n-2}^{(\alpha+1,\beta+1)}(x)
       .
    \end{align*}
    Analogously, for the odd part we have
    \begin{align*}
        \frac{f(x;\mu)-f(x;-\mu)}{2}
        &=
        A \sum_{k=0}^{[(n-1)/2]} \mu^{2k+1}
            \frac{d^{2k+1}}{dx^{2k+1}}P_n^{(\alpha,\beta)}(x)
            + \sum_{k=0}^{[(n-1)/2]} \mu^{2k+1}
            \frac{d^{2k}}{dx^{2k}}P_{n-1}^{(\alpha,\beta)}(x)
        \\
        &= A \frac{n+\alpha+\beta+1}{2}
          \sum_{k=0}^{[(n-1)/2]} \mu^{2k+1}
            \frac{d^{2k}}{dx^{2k}}P_{n-1}^{(\alpha+1,\beta+1)}(x)
            + \sum_{k=0}^{[(n-1)/2]} \mu^{2k+1}
            \frac{d^{2k}}{dx^{2k}}P_{n-1}^{(\alpha,\beta)}(x).
    \end{align*}
    The same manipulations with~$g(x;\mu)$ give
    \begin{align*}
        \frac{g(x;\mu)+g(x;-\mu)}{2}
        &=
        \sum_{k=0}^{[(n-1)/2]} \mu^{2k+2} \frac{d^{2k+1}}{dx^{2k+1}}P_n^{(\alpha,\beta)}(x)
        + A \sum_{k=0}^{[(n-1)/2]} \mu^{2k} \frac{d^{2k}}{dx^{2k}}P_{n-1}^{(\alpha,\beta)}(x)
        \\
        &= \frac{n+\alpha+\beta+1}{2}\mu^2 \sum_{k=0}^{[(n-1)/2]} \mu^{2k}
            \frac{d^{2k}}{dx^{2k}}P_{n-1}^{(\alpha+1,\beta+1)}(x)
         + A \sum_{k=0}^{[(n-1)/2]} \mu^{2k} \frac{d^{2k}}{dx^{2k}}P_{n-1}^{(\alpha,\beta)}(x)
        \quad\text{and}\\
        \frac{g(x;\mu)-g(x;-\mu)}{2}
        &=
        \sum_{k=0}^{[n/2]} \mu^{2k+1}
            \frac{d^{2k}}{dx^{2k}}P_n^{(\alpha,\beta)}(x)
            + A \sum_{k=0}^{[(n-1)/2]} \mu^{2k+1}
            \frac{d^{2k+1}}{dx^{2k+1}}P_{n-1}^{(\alpha,\beta)}(x)
        \\
        &= \mu \left(
          \sum_{k=0}^{[n/2]} \mu^{2k}
            \frac{d^{2k}}{dx^{2k}}P_{n}^{(\alpha,\beta)}(x)
            +  A \frac{n+\alpha+\beta}{2} \sum_{k=0}^{[n/2]-1} \mu^{2k}
            \frac{d^{2k}}{dx^{2k}}P_{n-2}^{(\alpha+1,\beta+1)}(x)\right).
    \end{align*}
    The polynomials~$f(1;\mu)$ and~$g(1;\mu)$ are stable by Lemma~\ref{lemma:f_stable} and
    Lemma~\ref{lemma:g_stable}, respectively. Thus, the pairs of polynomials mentioned
    in~\eqref{eq:r_comb_2} and~\eqref{eq:r_comb_3} have interlacing zeros by
    Theorem~\ref{th:H-B}.
\end{proof}

\begin{lemma}[see e.g.~{\cite[Lemma 3.4]{Wa}} or~{\cite[Lemma 3.5]{CW}}]\label{lemma:2pair_interl}
    Let real polynomials~$p(x)$ and~$q(x)$ such that $p(0),q(0)>0$ have only negative zeros.
    Then $\left(p(x),xq(x)\right)$ is a real pair if and only if the combinations
    \begin{equation}\label{eq:r_comb_pq}
      r_1(x)\colonequals r_1(x;A,B)\colonequals Ap(x)+Bxq(x)
        \quad\text{and}\quad
      r_2(x)\colonequals r_2(x;C,D)\colonequals Cp(x)+Dq(x)
    \end{equation}
    are nonzero outside the real line for all~$A,B,C,D>0$.
\end{lemma}
Recall that $p(x)$ and~$xq(x)$ is a real pair whenever they interlace (non-strictly if $p(x)$
and~$xq(x)$ have a common zero). For $p$ and~$q$ as in this lemma we thus have
$\deg q\leqslant\deg p\leqslant1+\deg q$ automatically.
\begin{proof}
    Without loss of generality, assume in the proof that the polynomials~$p$ and~$q$ have no
    common zeros: if not, the zeros are real and we can factor them out of~$r_1$ and $r_2$. The
    presence of common zeros prevents $p$ and~$q$ from being strictly interlacing.

\setstretch{1.05}
    Let $\left(p(x),xq(x)\right)$ be a real pair. The polynomials~$p$ and~$xq$ interlace exactly
    when between each pair of consecutive zeros~$\zr_i(p),\zr_{i-1}(p)$ the polynomial~$q$ has
    exactly one change of sign, $i=2,\dots,\deg p$. That is, the interlacing property of these
    polynomials is equivalent to
    \[
      R(z)\colonequals\frac{q(z)}{p(z)}=\gamma+\sum_{i=1}^{n}\frac{A_i}{z-\zr_i(p)}
      \quad\text{implying}\quad
      zR(z) = \frac{zq(z)}{p(z)}
        =\gamma z+\sum_{i=1}^{n}A_i + \sum_{i=1}^{n}\frac{\zr_i(p)A_i}{z-\zr_i(p)},
    \]
    where $\gamma\ge 0$ and the residues~$A_i=\frac{q(\zr_i(p))}{p'(\zr_i(p))}$ are positive for
    all~$i$. The straightforward check shows that
    \[\sign\Im R(z)=-\sign\Im \left(zR(z)\right)=-\sign\Im z\]
    for any~$z\in\mathbb{C}$ such that $p(z)\ne 0$.
    Consequently, the combinations~\eqref{eq:r_comb_pq} have zeros only on the real line.

    Conversely, let for any fixed $A,B,C,D>0$ the polynomials~$r_1(x;A,B)$ and~$r_2(x;C,D)$ have
    only real zeros. The zeros of~$r_1$, $r_2$ are all negative since all their coefficients are
    positive. Moreover, $p(x)$ and $q(x)$ are coprime, and therefore
    \[
    \begin{aligned}
        r_1(x_*;A,B)&=0&&\implies&&\sign p(x_*)=\sign q(x_*)=\sign r_2(x_*;C,D)\ne 0
        &&\text{and}&&
        r_1(x_*;\tilde{A},\tilde{B})\ne 0,\\
        r_2(x_\#;C,D)&=0&&\implies&&\sign p(x_\#)=-\sign q(x_\#)=\sign r_1(x_\#;A,B)\ne 0
        &&\text{and}&&
        r_2(x_\#;\tilde{C},\tilde{D})\ne 0,
    \end{aligned}
    \]
    where~$B\tilde{A}\ne A\tilde{B}$ and~$D\tilde{C}\ne C\tilde{D}$.

    The roots of a polynomial depend continuously on its coefficients. Therefore, when the
    ratio~$B/A$ comes close to zero (or to infinity), the roots of~$r_1$ tend to the roots
    of~$p(x)$ (or to the roots of~$xq(x)$, respectively). For $i=1,\dots,\deg r_1$ and $A,B>0$
    the zero $\zr_i(r_1)$ can never coincide with a root of~$p(x)$ or~$xq(x)$ thus remaining in
    the interval
    \[
    I_i\colonequals\bigcup_{A/B>0}\zr_i(r_1)
    =\left(\min\left\{\zr_i(p),\zr_{i-1}(q)\right\},
        \max\left\{\zr_i(p),\zr_{i-1}(q)\right\}\right).
    \]
    When~$k\colonequals\deg p-\deg q-1 \ne 0$ there exist~$|k|$ surplus roots of~$r_1(x)$ which
    disappear as~$r_1(x)$ becomes proportional to~$p(x)$ or~$xq(x)$. Being negative, these roots
    must tend to~$-\infty$. Since they can never meet a root of~$p(x)$ or~$q(x)$, they run the
    whole ray~$\left(-\infty,\zr_{\deg p}(p)\right)$ if $k>0$ and
    $\left(-\infty,\zr_{\deg q}(q)\right)$ if $k<0$. This implies~$|k|\le 1$,
    because~$r_1(x;A,B)$ and $r_1(x;\tilde{A},\tilde{B})$ have no common zeros
    unless~$B/A = \tilde{B}/\tilde{A}$.
    
    For the polynomial~$r_2$ we analogously obtain
    \[
    J_i\colonequals\bigcup_{C/D>0}\zr_i(r_2)
    =\left(\min\left\{\zr_i(q),\zr_i(p)\right\},
        \max\left\{\zr_i(q),\zr_i(p)\right\}\right)
    \]
    and~$|\deg p-\deg q|\le 1$. The last inequality together with~$|k|\le 1$ gives
    \( \deg q \le \deg p \le \deg q +1 \). For each~$z<0$ there exists $i=1,\dots,\deg q +1$
    such that~$z\in\left[\zr_i(q), \zr_{i-1}(q)\right)\subset \overline{J_i}\cup I_i$. Since
    $J_i$ and $I_i$ are disjoint, the only option is $\zr_i(q)\le \zr_i(p)\le \zr_{i-1}(q)$.
    This implies that the polynomial~$p(x)$ interlace~$xq(x)$ since $\deg p \le \deg q +1$.
\end{proof}

The next corollary complements the interlacing property of the polynomials
$\phi_n^{(\alpha,\beta)}(\mu)$ and $\phi_{n-1}^{(\alpha+1,\beta+1)}(\mu)$ (see the remark to
Theorem~\ref{th:Cs}).
\begin{corollary}\label{cr:interl_2}
    If $-1<\alpha<0$ and $\beta>0$ the pairs 
    $\left(\phi_n^{(\alpha,\beta)}(\mu), \mu\phi_{n-2}^{(\alpha+1,\beta+1)}(\mu)\right)$ and
    $\left(\phi_{n}^{(\alpha,\beta)}(\mu), \mu\phi_{n}^{(\alpha+1,\beta+1)}(\mu)\right)$ possess
    the (strict) interlacing property.
\end{corollary}
\begin{proof}
    By Theorem~\ref{th:Cs}, all involved polynomials have only real nonpositive zeros.
    Corollary~\ref{cr:main_2} adds that the polynomials
    in~\eqref{eq:r_comb_2}--\eqref{eq:r_comb_3} have (strictly) interlacing zeros. Therefore,
    Lemma~\ref{lemma:2pair_interl} assures the asserted fact.
\end{proof}
\begin{theorem}\label{thm:cb_main_2}
    If $-1<\alpha<0<\beta$ and $n=5,6,\dots$, then the polynomial
    $\phi_n^{(\alpha+1,\beta+1)}(\mu)$ interlaces $\phi_{n-2}^{(\alpha+1,\beta+1)}(\mu)$, and
    the polynomial $\phi_n^{(\alpha,\beta)}(x)$ interlaces $\phi_{n-2}^{(\alpha,\beta)}(x)$.
\end{theorem}
\begin{proof}
    According to Corollary~\ref{cor:main}, we have
    \begin{equation}\label{eq:z_rel0}
    \begin{aligned}
        &
        \zr_i\left(\phi_{n-2}^{(\alpha+1,\beta+1)}\right)
        <\zr_i\left(\phi_{n-1}^{(\alpha+1,\beta+1)}\right)
        <\zr_{i-1}\left(\phi_{n-2}^{(\alpha+1,\beta+1)}\right),
        \\
        &
        \zr_i\left(\phi_{n-1}^{(\alpha+1,\beta+1)}\right)
        <\zr_i\left(\phi_{n}^{(\alpha+1,\beta+1)}\right)
        <\zr_{i-1}\left(\phi_{n-1}^{(\alpha+1,\beta+1)}\right)
    \end{aligned}
    \end{equation}
    for any natural~$i\leqslant n/2$. From Corollary~\ref{cr:interl_2} we obtain that
    \begin{align}\label{eq:cz_rel3}
        &
        \zr_i\left(\phi_{n-2}^{(\alpha+1,\beta+1)}\right)
        <\zr_i\left(\phi_{n}^{(\alpha,\beta)}\right)
        <\zr_{i-1}\left(\phi_{n-2}^{(\alpha+1,\beta+1)}\right),
        \\
        &
        \label{eq:cz_rel5}
        \zr_i\left(\phi_{n}^{(\alpha+1,\beta+1)}\right)
        <\zr_i\left(\phi_{n}^{(\alpha,\beta)}\right)
        <\zr_{i-1}\left(\phi_{n}^{(\alpha+1,\beta+1)}\right).
    \end{align}
    
    Bringing together the right inequality in~\eqref{eq:cz_rel3} and the left inequalities
    in~\eqref{eq:cz_rel5} and~\eqref{eq:z_rel0}, we obtain
    \begin{equation} \label{eq:z_rel1}
        \zr_i\left(\phi_{n-2}^{(\alpha+1,\beta+1)}\right)
        \overset{\eqref{eq:z_rel0}}<\zr_i\left(\phi_{n-1}^{(\alpha+1,\beta+1)}\right)
        \overset{\eqref{eq:z_rel0}}<\zr_i\left(\phi_{n}^{(\alpha+1,\beta+1)}\right)
        \overset{\eqref{eq:cz_rel5}}<\zr_i\left(\phi_{n}^{(\alpha,\beta)}\right)
        \overset{\eqref{eq:cz_rel3}}<\zr_{i-1}\left(\phi_{n-2}^{(\alpha+1,\beta+1)}\right)
    \end{equation}
    for all natural~$i\leqslant n/2$. This relation implies that the zeros of
    $\phi_{n}^{(\alpha+1,\beta+1)}(\mu)$ and~$\phi_{n-2}^{(\alpha+1,\beta+1)}(\mu)$ interlace.

    By Corollary~\ref{cor:main} we obtain (cf.~\eqref{eq:z_rel0})
    \begin{equation*}
        \zr_i\left(\phi_{n-2}^{(\alpha,\beta)}\right)
        <\zr_i\left(\phi_{n-1}^{(\alpha,\beta)}\right)
        <\zr_i\left(\phi_{n}^{(\alpha,\beta)}\right),\quad i=1,2,\dots.
    \end{equation*}
    This chain can be continued with the left inequality in~\eqref{eq:cz_rel5} and the right
    inequality in~\eqref{eq:cz_rel3} so that
    \[
        \zr_i\left(\phi_{n-2}^{(\alpha+1,\beta+1)}(\mu)\right)
        \overset{\eqref{eq:cz_rel5}}<\zr_i\left(\phi_{n-2}^{(\alpha,\beta)}\right)
        <\zr_i\left(\phi_{n-1}^{(\alpha,\beta)}\right)
        <\zr_i\left(\phi_{n}^{(\alpha,\beta)}\right)
        \overset{\eqref{eq:cz_rel3}}<\zr_{i-1}\left(\phi_{n-2}^{(\alpha+1,\beta+1)}\right)
    \]
    for each natural~$i$.
\end{proof}

\section{Conclusion: relations between Conjecture~\texorpdfstring{\ref{con:main}}{A} and
    Conjecture~\texorpdfstring{\ref{con:main_2}}{B}}\label{sec:concl-relat-conj}
Here we obtain two instructive facts giving an idea about the limits of the current approach. In
the present note, we used a modification of the method applied in~\cite{CW}, so it has the same
deficiency: the parameters~$\alpha$ and~$\beta$ are constrained to provide the positivity
of~$w'_{\alpha,\beta}(x)$ and the convergence of integrals. However, these sufficient conditions
seem to be quite far from being necessary.

The following two lemmata coupled with Theorem~\ref{thm:cb_main_2} give no new parameter range
for Conjecture~\ref{con:main} to hold. At the same time, the comparison to Lemma~\ref{lem:main}
clearly shows that this conjecture is less restrictive than Conjecture~\ref{con:main_2}.

\begin{lemma}\label{lemma:instr1}
    If the polynomials $\phi_n^{(\alpha,\beta)}(\mu)$, $\phi_{n-1}^{(\alpha,\beta)}(\mu)$ and
    $\phi_{n-2}^{(\alpha,\beta)}(\mu)$ are pairwise interlacing in such a way that
    \[
    \zr_1\left(\phi_{n-2}^{(\alpha,\beta)}\right)
    <\zr_1\left(\phi_{n-1}^{(\alpha,\beta)}\right)
    <\zr_1\left(\phi_n^{(\alpha,\beta)}\right),
    \]
    then $\phi_n^{(\alpha,\beta-1)}(\mu)$ interlaces $\phi_{n-1}^{(\alpha,\beta-1)}(\mu)$.
\end{lemma}
\begin{proof}
    By the formula~\eqref{eq:fl_summ1}, $\phi_n^{(\alpha,\beta-1)}(\mu)$ and
    $\phi_{n-1}^{(\alpha,\beta-1)}(\mu)$ have only real zeros. Furthermore,
    \[
    \zr_{i+1}\left(\phi_n^{(\alpha,\beta)}\right)
    <\zr_i\left(\phi_{n-2}^{(\alpha,\beta)}\right)
    <\zr_i\left(\phi_{n-1}^{(\alpha,\beta-1)}\right)
    <\zr_i\left(\phi_{n-1}^{(\alpha,\beta)}\right)
    <\zr_i\left(\phi_n^{(\alpha,\beta-1)}\right)
    <\zr_i\left(\phi_n^{(\alpha,\beta)}\right)
    \]
    for natural~$i=1,\dots,\big[\frac{n-1}2\big]$. This implies the interlacing property for the polynomials
    $\phi_n^{(\alpha,\beta-1)}(\mu)$ and~$\phi_{n-1}^{(\alpha,\beta-1)}(\mu)$.
\end{proof}

As an intermediate result (Corollary~\ref{cr:interl_2}) we had the interlacing property of
$\phi_n^{(\alpha,\beta)}(\mu)$ and $\phi_{n}^{(\alpha+1,\beta+1)}(\mu)$ when the parameters
satisfy~$-1<\alpha<0<\beta$. Such a fact allows us to get a relationship complementing
Lemma~\ref{lemma:instr1}.
\begin{lemma}
    Let the polynomial pairs
    $\left(\phi_n^{(\alpha,\beta)}(\mu),\phi_{n}^{(\alpha+1,\beta+1)}(\mu)\right)$, \
    $\left(\phi_n^{(\alpha,\beta)}(\mu),\phi_{n-1}^{(\alpha+1,\beta+1)}(\mu)\right)$ be
    interlacing in such a way that
    \begin{equation}\label{eq:zrel_mu}
        \zr_1\left(\phi_{n-1}^{(\alpha+1,\beta+1)}\right)
        <\zr_1\left(\phi_{n}^{(\alpha,\beta)}\right),
        \quad
        \zr_1\left(\phi_{n}^{(\alpha+1,\beta+1)}\right)
        <\zr_1\left(\phi_{n}^{(\alpha,\beta)}\right).
    \end{equation}
    Then $\phi_n^{(\alpha+1,\beta)}(\mu)$ interlaces $\phi_{n-1}^{(\alpha+1,\beta)}(\mu)$.
\end{lemma}
\begin{proof}
    The identities~\eqref{eq:fl_summ1} and~\eqref{eq:fl_summ2} give
    \begin{align}\label{eq:zrel_x1}
        (n+\alpha+\beta+2)\phi_n^{(\alpha+1,\beta+1)}+(n+\alpha+1)\phi_{n-1}^{(\alpha+1,\beta+1)}
        &=(2n+\alpha+\beta+2) \phi_{n}^{(\alpha+1,\beta)}
        ,\\\label{eq:zrel_x2}
        (n+\alpha+\beta+1)\phi_{n}^{(\alpha+1,\beta)}-(2n+\alpha+\beta+1) \phi_{n}^{(\alpha,\beta)}
        &=(n+\beta)\phi_{n-1}^{(\alpha+1,\beta)}
        .
    \end{align}
    From the inequalities~\eqref{eq:zrel_mu} we obtain that each interval
    $\left(\zr_{i+1}\left(\phi_{n}^{(\alpha,\beta)}\right),
        \zr_i\left(\phi_{n}^{(\alpha,\beta)}\right)\right)$, \ $i=1,\dots,[n/2]-1$, contains the
    points~$\zr_i\left(\phi_{n-1}^{(\alpha+1,\beta+1)}\right)$
    and~$\zr_i\left(\phi_{n}^{(\alpha+1,\beta+1)}\right)$ and no other zeros of the
    polynomials~$\phi_{n-1}^{(\alpha+1,\beta+1)}$ and~$\phi_{n}^{(\alpha+1,\beta+1)}$. Thus, the
    left-hand side of~\eqref{eq:zrel_x1} also has exactly one zero on each of the intervals. As
    a result, $\phi_{n}^{(\alpha+1,\beta)}$ interlaces~$\phi_{n}^{(\alpha,\beta)}$, so the zeros
    of their difference appearing in~\eqref{eq:zrel_x2} and hence
    of~$\phi_{n-1}^{(\alpha+1,\beta)}$ are interlacing with the zeros
    of~$\phi_{n}^{(\alpha+1,\beta)}$.
\end{proof}

\section*{Acknowledgments}
The authors thank Olga Holtz, Ren\'e van Bevern, Olga Kushel and an anonymous referee for their
attention and valuable remarks.


\vspace{1em}
\makebox[21ex][l]{Alexander Dyachenko}
\texttt{{\href{mailto:dyachenk@math.tu-berlin.de}{dyachenk@math.tu-berlin.de},
        \href{mailto:diachenko@sfedu.ru}{diachenko@sfedu.ru}}}\\[.2em]
{\itshape
    \makebox[21ex][l]{}
    TU-Berlin, Institut f\"ur Mathematik, Sekr.~MA 4-2\\
    \makebox[21ex][l]{}
    Straße des 17. Juni 136, 10623 Berlin, Germany.
}\\[1em]
\makebox[21ex][l]{Galina van Bevern}
\texttt{\href{mailto:gvbevern@yandex.com}{gvbevern@yandex.com}}\\[.2em]
{\itshape
    \makebox[21ex][l]{}
    TU-Berlin, Institut f\"ur Mathematik, Sekr.~MA 4-2\\
    \makebox[21ex][l]{}
    Straße des 17. Juni 136, 10623 Berlin, Germany.
}\\[.2em]
{\itshape
    \makebox[21ex][l]{}
    (Current Address:
    Tomsk Polytechnic University, Institute of Physics and Technology\\
    \makebox[21ex][l]{}
    Department of Higher Math.\ and Math.\ Phys., 
    Lenin Avenue 2/A, 634000 Tomsk, Russia.)}

\end{document}